\documentclass[12pt]{article}
\usepackage{graphicx}
\usepackage{amsthm,amsmath,amsfonts,amssymb} 
\usepackage{color}
\catcode`\@=11 \@addtoreset{equation}{section}

\catcode`\@=12

\newtheorem{Theorem}{Theorem}[section]
\newtheorem{Proposition}{Proposition}[section]
\newtheorem{Lemma}{Lemma}[section]
\newtheorem{Corollary}{Corollary}[section]
\newtheorem{Remark}{Remark}[section]

\newcommand{\bTheorem}[1]{
\begin{Theorem} \label{T#1} }
\newcommand{\eT}{\end{Theorem}}

\newcommand{\bProposition}[1]{
\begin{Proposition} \label{P#1}}
\newcommand{\eP}{\end{Proposition}}

\newcommand{\bRemark}[1]{
\begin{Remark} \label{R#1}}
\newcommand{\eR}{\end{Remark}}

\newcommand{\bLemma}[1]{
\begin{Lemma} \label{L#1} }
\newcommand{\eL}{\end{Lemma}}

\newcommand{\bCorollary}[1]{
\begin{Corollary} \label{C#1} }
\newcommand{\eC}{\end{Corollary}}

\newcommand{\bFormula}[1]{
\begin{equation} \label{#1}}
\newcommand{\eF}{\end{equation}}

\newcommand{\DC}{C^\infty_c}

\newcommand{\vu}{\vc{u}}
\newcommand{\vc}[1]{{\bf #1}}
\newcommand{\Div}{{\rm div}_x}
\newcommand{\Grad}{\nabla_x}

\newcommand{\tn}[1]{\mbox {\F #1}}
\newcommand{\dx}{{\rm d} {x}}
\newcommand{\ds}{{\rm d} {s}}

\newcommand{\bg}{\begin{equation}}
\newcommand{\ed}{\end{equation}}
\newcommand{\bga}{\begin{eqnarray}}
\newcommand{\eda}{\end{eqnarray}}

\newcommand{\intR}[1]{\int_{\mathbb R^3} #1 \ \dx}

\newcommand{\E}{\mathcal{E}}
\newcommand{\Aa}{\mathcal{A}}

\newcommand{\ep}{\varepsilon}

\font\F=msbm10 scaled 1200

\newcommand{\Del}{\Delta_x}

\newcommand{\RR}{{\mathbb R}}

\DeclareMathOperator{\trace}{trace}

\definecolor{Cgrey}{rgb}{0.85,0.85,0.85}
\definecolor{Cblue}{rgb}{0.50,0.85,0.85}
\definecolor{Cred}{rgb}{1,0,0}
\definecolor{fancy}{rgb}{0.10,0.85,0.10}

\newcommand\Cbox[2]{%
    \newbox\contentbox%
    \newbox\bkgdbox%
    \setbox\contentbox\hbox to \hsize{%
        \vtop{
            \kern\columnsep
            \hbox to \hsize{%
                \kern\columnsep%
                \advance\hsize by -2\columnsep%
                \setlength{\textwidth}{\hsize}%
                \vbox{
                    \parskip=\baselineskip
                    \parindent=0bp
                    #2
                }%
                \kern\columnsep%
            }%
            \kern\columnsep%
        }%
    }%
    \setbox\bkgdbox\vbox{
        \color{#1}
        \hrule width  \wd\contentbox %
               height \ht\contentbox %
               depth  \dp\contentbox
        \color{black}
    }%
    \wd\bkgdbox=0bp%
    \vbox{\hbox to \hsize{\box\bkgdbox\box\contentbox}}%
    \vskip\baselineskip%
}

\usepackage{color}

\begin{document}


\title{On asymptotic isotropy for a hydrodynamic model \\ of liquid crystals}

\author{Mimi Dai \and Eduard Feireisl \thanks{The research of E.F. leading to these results has received funding from the European Research Council under the European Union's 
Programme (FP7/2007-2013)/ ERC Grant Agreement 320078. The Institute of Mathematics of the Academy of Sciences of the Czech
Republic is supported by RVO:67985840.}
\and Elisabetta Rocca
\thanks{The work of E.R. was supported by the
 FP7-IDEAS-ERC-StG \#256872 (EntroPhase) and by GNAMPA (Gruppo Nazionale per l'Analisi Matematica, la Probabilit\`a e le loro Applicazioni) of INdAM (Istituto Nazionale di Alta Matematica).}\and Giulio Schimperna \and Maria E. Schonbek \thanks{The
 research of M.S. was partially supported by NSF Grant DMS-0900909.}}




\maketitle

\bigskip
\centerline{Department of Mathematics, University of Illinois at Chicago}
\centerline{851 S. Morgan Street
Chicago, IL 60607-7045, USA}

\bigskip
\centerline{Institute of Mathematics of the Academy of Sciences of the Czech Republic}
\centerline{\v Zitn\' a 25, 115 67 Praha 1, Czech Republic}

\bigskip
\centerline{Weierstrass Institute for Applied
Analysis and Stochastics}
\centerline{Mohrenstr.~39, D-10117 Berlin,
Germany}
\centerline{and}
\centerline{Dipartimento di Matematica, Universit\`a di Milano}
\centerline{Via Saldini 50, 20133 Milano, Italy}

\bigskip
\centerline{Dipartimento di Matematica, Universit\`a di Pavia}
\centerline{Via Ferrata 1, I-27100 Pavia, Italy}

\bigskip
\centerline{Department of Mathematics, University of California}
\centerline{Santa Cruz, CA 95064, USA}



\bigskip

\bigskip

\bigskip

\begin{abstract}
We study a PDE system describing the motion of liquid crystals by means of the $Q-$tensor description
for the crystals coupled with the incompressible Navier-Stokes system. Using the method of Fourier
splitting, we show that solutions of the system tend to the isotropic state at the rate $(1 + t)^{-3/2}$
as $t \to \infty$.
\end{abstract}

\medskip

{\bf Key words:} Liquid crystal, $Q-$tensor description, long-time behavior, Fourier splitting

\medskip

\tableofcontents

\section{Introduction}
\label{i}

We consider a frequently used hydrodynamic model of
nematic liquid crystals, where the local configuration of the crystal is represented
by the $Q-$tensor $\tn{Q} = \tn{Q}(t,x)$, while its motion is described
through the Eulerian velocity field $\vu = \vu(t,x)$, both quantities being functions of the
time $t > 0$ and the spatial position $x \in \RR^3$.
The tensor $\tn{Q} \in \RR^{3 \times 3}_{{\rm sym},0}$
is a symmetric traceless matrix, whose time evolution is described by the equation

\Cbox{Cgrey}{

\bFormula{i1}
\partial_t \tn{Q} + \Div (\tn{Q} \vu ) - \tn{S}(\Grad \vu, \tn{Q} ) = \Del \tn{Q} - \mathcal{L} [ \partial F (\tn{Q}) ],
\eF

}

\noindent
with
\[
\mathcal{L}[\tn{A}] \equiv \tn{A} - \frac{1}{3} {\rm tr}[\tn{A}] \tn{I}
\]
denoting the projection onto the space of traceless matrices, and $F$ denoting a potential function which will be described later.
The velocity field obeys the Navier-Stokes system

\Cbox{Cgrey}{

\bFormula{i2}
\partial_t \vu + \Div (\vu \otimes \vu ) + \Grad p = \Del \vu + \Div \Sigma (\tn{Q})
\eF
supplemented with the incompressibility constraint
\bFormula{i3}
\Div \vu = 0.
\eF

}

The tensors $\tn{S}$ and $\Sigma$ are taken the form

\Cbox{Cgrey}{

\bFormula{i4}
\tn{S}(\Grad \vu, \tn{Q}) = (\xi \ep(\vu) + \omega(\vu)) \left( \tn{Q} + \frac{1}{3} \tn{I} \right) + \left( \tn{Q} + \frac{1}{3} \tn{I} \right)
(\xi \ep(\vu) - \omega(\vu)) - 2 \xi \left( \tn{Q} + \frac{1}{3} \tn{I} \right) \tn{Q} : \Grad \vu,
\eF

\bFormula{i5}
\Sigma (\tn{Q} ) = 2 \xi \tn{H} : \tn{Q} \left( \tn{Q} + \frac{1}{3} \tn{I} \right) - \xi \left[
\tn{H} \left( \tn{Q} + \frac{1}{3} \tn{I} \right) - \left( \tn{Q} + \frac{1}{3} \tn{I} \right) \tn{H} \right] - (\tn{Q} \tn{H} - \tn{H} \tn{Q} ) -
\Grad \tn{Q} \odot \Grad \tn{Q},
\eF

}

\noindent where we have denoted
\[
 \epsilon (\vu) = \frac{1}{2} ( \Grad \vu + \Grad^t \vu ), \
  \omega (\vu) =  \frac{1}{2} ( \Grad \vu - \Grad^t \vu ),
\]
\[
 \tn{H} = \Del \tn{Q} - \mathcal{L} [ \partial F (\tn{Q}) ], \ \mbox{and}\
 (\Grad \tn{Q} \odot \Grad \tn{Q})_{ij} = \partial_i \tn{Q}_{\alpha\beta} \partial_j \tn{Q}_{\alpha\beta}.
\]
Here and hereafter, we use the summation convention for repeated indices. The number $\xi \in \RR$ is a
scalar parameter measuring the ratio between the rotation and the aligning effect that a shear
flow exerts over the directors.

We refer to Beris and Edwards \cite{BerEdw} for the physical background, and to Zarnescu et al. \cite{KirWiZa},
\cite{PaiZar2}, \cite{PaiZar1} for mathematical aspects of the problem.


\subsection{Energy balance}

The problem (\ref{i1} - \ref{i5}) admits a natural \emph{energy functional}, namely
\[
E = \frac{1}{2} |\vu|^2 + \frac{1}{2} |\Grad \tn{Q} |^2 + F(\tn{Q}),
\]
where
\[
F: \tn{R}^{3 \times 3}_{\rm sym} \to (- \infty, \infty]
\]
is a given (generalized) function.

We assume that $F \in C^2(\mathcal{O})$, where
$\mathcal{O} \subset \RR^{3 \times 3}_{\rm sym}$ is an open set containing the isotropic
state $\tn{Q} \equiv 0$, and there are two balls $B_{r_1}$, $B_{r_2}$
with $r_1 < r_2$ such that
\[
\tn{Q} = 0 \in B_{r_1} \equiv \{ |\tn{Q}| < r_1 \} \subset B_{r_2} \equiv \{ |\tn{Q}| \le r_2 \} \subset \mathcal{O}.
\]
In addition, we suppose that $\tn{Q} = 0$ is the (unique) global minimum of $F$ in $\mathcal{O}$, specifically,
\bFormula{i6}
  F(0) = 0, \ F(\tn{Q}) > 0 \ \mbox{for any} \ \tn{Q} \in \mathcal{O} \setminus \{0\}
\eF
and
\bFormula{i7}
\partial F(\tn{Q}) : \tn{Q} \geq 0 \ \mbox{whenever} \ \tn{Q} \in B_{r_1} \ \mbox{or} \ \tn{Q} \in \mathcal{O} \setminus B_{r_2}.
\eF

Let us note here that the polynomial potentials considered by Paicu and Zarnescu \cite{PaiZar2}:
\bFormula{D1intro}
F(\tn{Q}) = \frac{a}{2} |\tn{Q}|^2 + \frac{b}{3} \trace[\tn{Q}^3 ] + \frac{c}{4} |\tn{Q}|^4,
\eF
at least in case $a>0$ in a neighborhood of 0, fit this conditions (cf. Section~\ref{D} for further comments on this point).

Taking the scalar product of equation (\ref{i1}) with $\tn{H}$, the scalar product of
equation (\ref{i2}) with $\vu$, adding the
resulting expressions and integrating over the physical space $\RR^3$
(cf.~\cite[Proof of Prop.~1]{PaiZar1} for details),
we obtain the total energy balance
\bFormula{i8}
 \frac{{\rm d}}{{\rm d}t} \intR{ E } + \intR{ |\Grad \vu |^2 + \Big| \Del \tn{Q}
  - \mathcal{L}[\partial F(\tn{Q})] \Big|^2 } = 0
\eF
provided that
\bFormula{i9}
\vu \to 0, \ \tn{Q} \to 0 \ \mbox{as}\ |x| \to \infty
\eF
sufficiently fast.

The presence of the dissipative term
\[
\intR{ |\Grad \vu |^2 + \Big| \Del \tn{Q} - \mathcal{L}[\partial F(\tn{Q})] \Big|^2 }
\]
suggests that
\bFormula{i10-}
\Grad \vu(t, \cdot) \to 0,\ \tn{Q}(t, \cdot) \to \tilde {\tn{Q} } \ \mbox{as}\ t \to \infty \ \mbox{in a certain sense,}
\eF
where $\tilde {\tn Q}$ is a static distribution of the $Q-$tensor density, namely it
satisfies
\bFormula{i11}
- \Del \tilde{ \tn{Q} } + \mathcal{L} \left[ \partial F (\tilde{\tn{Q} }) \right] = 0.
\eF

As we shall see below (Lemma \ref{Lp1}), the hypothesis (\ref{i6}) implies that $\tilde {\tn{Q}} \equiv 0$;
more specifically, any solution $\tilde {\tn{Q}}$ of (\ref{i11}) belonging to the class
\[
  \Grad \tilde{\tn{Q}} \in L^2(\RR^3; \RR^3), \
   F(\tilde{\tn{Q}}) \in L^1(\RR^3), \ \tilde{\tn{Q}}(x) \in B_{r_2} \ \mbox{for all}\ x \in \RR^3
\]
necessarily vanishes identically in $\RR^3$, in particular (\ref{i10-}) reduces to
\bFormula{i10}
\Grad \vu(t, \cdot) \to 0,\ \tn{Q}(t, \cdot) \to  0 \ \mbox{as}\ t \to \infty.
\eF


\subsection{Asymptotic isotropy}

Our goal is to justify (\ref{i10}) in the class of \emph{weak} solutions
to the system (\ref{i1} - \ref{i3}).  To this end, we need
a simplifying assumption setting the parameter $\xi = 0$.
Hence, (\ref{i4}), (\ref{i5}) reduce to
\bFormula{i13}
\tn{S}(\Grad \vu, \tn{Q}) = \omega(\vu) \tn{Q} - \tn{Q} \omega (\vu),
\
\Sigma (\tn{Q} ) =  - \tn{Q} \Del \tn{Q}  + \Del \tn{Q} \tn{Q}
- \Grad \tn{Q} \odot \Grad \tn{Q},
\eF
where we have used
\[
 \tn{Q} \mathcal{L} [ \partial F (\tn{Q}) ] - \mathcal{L} [ \partial F (\tn{Q}) ] \tn{Q}
 = \tn{Q} \partial F (\tn{Q}) - \partial F (\tn{Q}) \tn{Q}
 = 0.
\]
Such an assumption simplifies considerably the analysis of the $Q-$tensor equation (\ref{i1}), in particular we may use its renormalized version
in order to deduce stability of the isotropic state in the space $L^\infty$.

\medskip

\noindent Our aim is to show that
\bFormula{i14}
  \| \vu(t, \cdot) \|_{L^2(\RR^3; \RR^3)}
    + \| \tn{Q}(t, \cdot) \|_{H^1(\RR^3; \RR^{3 \times 3})}
   \leq c (1 + t)^{-3/4}, \ t > 0
\eF
for \emph{any} weak solution of the problem (\ref{i1} - \ref{i3}), (\ref{i13}),
where the constant $c$ depends only on the initial data. Such a result seems
optimal,  as the decay coincides with that for the linear heat equation.
We would like to point out that our hypotheses (cf.~(\ref{i6}), in particular) are also optimal
for \emph{unconditional} convergence to an equilibrium. Indeed
one may conjecture, by analogy with the nowadays standard existence theory for
semilinear elliptic problems developed by Berestycki and Lions \cite{BerLio1}, \cite{BerLio2},
that the stationary problem (\ref{i11}) may admit a non-zero solution if $F < 0$ at some point.
Under these circumstances, convergence to a single stationary state is in general not
expected, cf.~\cite[Theorem 4.1]{EF21}.

In order to show (\ref{i14}) we make use of the method of \emph{Fourier splitting}\/
developed in \cite{Sch3}, \cite{Sch2},
\cite{Sch1} and later used in \cite{DaiQSch} to study the
long-time behavior of a liquid crystal model based on the
description via the director field. Besides the higher complexity of the $Q-$tensor model
reflected through the constitutive relations (\ref{i13}),
the main difference between \cite{DaiQSch} and this paper is that the present result is
\emph{unconditional} and applies to all weak solutions of the problem  satisfying
an energy inequality, while \cite{DaiQSch}
requires the initial data to be small and regular.
As is well known, the ultimate regularity of the Navier-Stokes and related problems is
based on the so-called Ladyzhenskaya estimates (cf.~\cite{DaiQSch}) available for the
present problem only in the $2D-$geometry, see Paicu and Zarnescu \cite{PaiZar2}.

The paper is organized as follows. In Section \ref{p}, we introduce the concept
of \emph{finite energy weak solution} to the problem (\ref{i1} - \ref{i3}),
(\ref{i13}) and collect some preliminary material, including the energy
inequality and its immediate implications. Section \ref{m} states rigorously
our main result. Section \ref{d} deals with the $Q-$tensor
equation, in particular, we deduce decay estimates for $\tn{Q}$
assuming higher integrability of the initial data. The proof of
the decay of the velocity field is completed in Section~\ref{v}
by means of the Fourier splitting method. Finally, we discuss the implications of our results for a special
class of polynomial potentials in Section \ref{D}.


\section{Preliminaries, weak solutions, energy inequality}
\label{p}

The expected regularity of the weak solutions is basically determined by
the energy balance (\ref{i8}).  More specifically,
we consider the weak solutions
$(\tn{Q}$, $\vu)$ belonging to the following class:

\begin{itemize}

\item[(a)]
\[
  \tn{Q} \in C_{\rm weak} ([0, T] ; L^2(\RR^3; \RR^{3\times 3})),
   \ \sup_{t \in [0,T]} \left( \| \tn{Q}(t, \cdot)  \|_{L^1 \cap L^\infty (\RR^3;\RR^{3\times 3})}
   + \| \tn{Q} \|_{W^{1,2}(\RR^3; \RR^{3\times 3})} \right) < \infty,
\]
\[
  \tn{Q} \in L^2(0,T; W^{2,2}(\RR^3;\RR^{3\times 3})),
   \ \tn{Q}(t, x) \in  B_{r_2} \
    \mbox{for all}\ t \in [0,T],\ \mbox{a.a.}\ x \in \RR^3,
\]
for any $T > 0$ ;
\item[(b)]
\[
   \vu \in C_{\rm weak} (0,T; L^2(\RR^3; \RR^3)), \ \Grad \vu \in L^2(0,T; L^2(\RR^3; \RR^{3\times 3}))
\]
for any $T > 0$.
\end{itemize}

The last condition in~(a) states that $\tn{Q}$ remains
separated from the
boundary of the domain $\mathcal{O}$ of $F$ (if $F$ is allowed to explode
near $\partial\mathcal{O}$).
This property is often referred to as
\emph{strict physicality}\/ of the $Q-$tensor configuration.
Such a property has been recently proved by Wilkinson
\cite{Wil} in the case where system~(\ref{i1} - \ref{i5}) is settled
in the unit torus and complemented with periodic boundary conditions.
The estimates performed below (cf.~in particular Subsec.~\ref{MP})
will imply, as a byproduct, that the same property holds also in the present
case. A rigorous proof of \emph{existence}\/ for weak solutions
to~(\ref{i1} - \ref{i5}) in the whole euclidean space $\RR^3$
was established by Paicu and Zarnescu \cite{PaiZar2},
\cite{PaiZar1} for a certain class of smooth potentials $F$.
Actually, the uniform estimates we are going to detail below
can give some idea on the highlights of their argument;
moreover, they will show that \emph{singular potentials}\/
satisfying (\ref{i6}), (\ref{i7}) can be dealt with by the same method.
We finally note that the regularity conditions stated in~(a)
are fully consistent with the a-priori estimates.


\subsection{Weak solutions for the $Q-$tensor equation}

If $\tn{Q}$, $\vu$ belong to the regularity class specified  above, it is easy to check that
\bFormula{p1}
  \| \tn{S}(\Grad \vu,\tn Q) \|_{L^2(0,T; L^1 \cap L^2 (\RR^3; \RR^{3 \times 3}))} \leq c,
\eF
\bFormula{p2}
  {\rm ess}\sup_{t \in (0,T)}
     \left\| \mathcal{L} [\partial F(\tn{Q}) ](t, \cdot) \right\|_{L^1 \cap L^\infty (\RR^3; \RR^{3 \times 3})} \leq c
\eF
and
\bFormula{p3}
  \Div (\tn{Q} \vu) = \vu \cdot \Grad \tn{Q} \in L^\infty (0,T; L^1 (\RR^3;\RR^{3\times 3}))
        \cap L^2 (0,T; L^{3/2}(\RR^3 ; \RR^{3\times 3}))
	\cap L^1(0,T; L^{3}(\RR^3 ; \RR^{3\times 3}))
\eF
for any $T > 0$, where we have used the embedding relation
$W^{1,2} \hookrightarrow L^6$ in three dimension.

Consequently, in view of the standard parabolic $L^p-L^q$ estimates, all
partial derivatives appearing in (\ref{i1}) exist in the strong sense and the
equation is satisfied a.e.~in the space time cylinder $[0, \infty) \times \RR^3$.


\subsection{The Navier-Stokes system}

As for the Navier-Stokes system (\ref{i2}), we have
\bFormula{p4}
   \vu \otimes \vu \in L^\infty (0,T; L^{1}(\RR^3; \RR^{3\times 3}))
      \cap L^2 (0,T; L^{3/2}(\RR^3; \RR^{3\times 3}))
      \cap L^1 (0,T; L^{3} (\RR^3;\RR^{3\times 3})),
\eF
\bFormula{p5}
\tn{Q} \Del \tn{Q}, \ \Del \tn{Q} \tn{Q} \in L^2(0,T; L^{1} \cap L^2 (\RR^3; \RR^{3\times 3})) ,
\eF
\bFormula{p51}
\Grad \tn{Q} \odot \Grad \tn{Q} \in
L^\infty (0,T; L^1 (\RR^3; \RR^{3\times 3})) \cap L^2 (0,T; L^2 (\RR^3; \RR^{3\times 3})),
\eF
where we have used the Gagliardo-Nirenberg interpolation inequality
\bFormula{p6}
\| \Grad v \|_{L^4}^2 \leq c \| \Del v \|_{L^2} \| v \|_{L^\infty}.
\eF

Thus, applying the standard \emph{Helmholtz projection} $\vc{P}$ onto the space of
solenoidal functions, the system (\ref{i2}) may be interpreted as a linear parabolic equation
\bFormula{p7}
\partial_t \vu - \Del \vu = \vc{P} \Div \Big[ - \tn{Q} \Del \tn{Q} + \Del \tn{Q} \tn{Q}  -
\Grad \tn{Q} \odot \Grad \tn{Q} - \vu \otimes \vu \Big],
\eF
with the right-hand side ranging in a Sobolev space $L^p(0,T;W^{-1,r}(\RR^3;\RR^3))$
for certain $p,r$.


\section{Main result}
\label{m}

We are ready to state the main result of the present paper.
\nopagebreak


\bTheorem{m1}
Let the potential $F$ satisfy the hypotheses (\ref{i6}), (\ref{i7}). Let $(\tn{Q}, \vu)$ be a
global-in-time weak solution of the system (\ref{i1} - \ref{i3}), (\ref{i13}) satisfying the energy inequality
\bFormula{m1}
\intR{ \left[ \frac{1}{2} |\vu|^2 + \frac{1}{2} |\Grad \tn{Q} |^2 + F(\tn{Q}) \right] (t, \cdot) }
+ \int_{s}^t \intR{ \left[ |\Grad \vu |^2 + \Big| \Del \tn{Q} - \mathcal{L}[\partial F(\tn{Q})] \Big|^2 \right] }
\eF
\[
\leq
\intR{ \left[ \frac{1}{2} |\vu|^2 + \frac{1}{2} |\Grad \tn{Q} |^2 + F(\tn{Q}) \right] (s, \cdot) }
\]
for all $t > s$ and a.a.~$s \in [0, \infty)$ including $s=0$, emanating from the initial data
\[
\vu(0, \cdot) = \vu_0, \ \tn{Q}(0, \cdot) = \tn{Q}_0,
\]
\bFormula{m2}
 \vu_0 \in L^1 \cap L^2 (\RR^3; \RR^3), \ \Div \vu_0  = 0,
 \ \tn{Q}_0 \in L^1\cap W^{1,2} (\RR^3; \RR^{3\times 3}_{{\rm sym},0} ),
  \ |\tn{Q}_0(x)| \le r_2 \ \mbox{for a.a.}\ x \in \RR^3,
\eF
where $r_2$ has been introduced in (\ref{i7}).

Then there exist a constant $c>0$ depending solely on the initial data
$[\vu_0, \tn{Q}_0]$ such that
\bFormula{m3}
   \| \vu(t, \cdot) \|_{L^2(\RR^3;\RR^3)}
     + \| \tn{Q}(t, \cdot) \|_{W^{1,2}(\RR^3; \RR^{3 \times 3})}  \leq c (1 + t)^{-\frac{3}{4}}
\eF
for all $t > 0$. If, in addition,
\bFormula{m3bis}
F(\tn{Q}) \geq \lambda |\tn{Q}|^2 \ \mbox{in}\ B_{r_1}, \ \lambda > 0,
\eF
then the decay  of the $L^2$ norm of $\tn{Q}$ is
\bFormula{d12bis-0}
  \| \tn{Q} (t, \cdot) \|_{L^2(\RR^3; \RR^{3\times 3})}^2
    \leq c \exp(-d t) \ \mbox{for all} \ t \geq 0 \ \mbox{and some}\ d > 0.
\eF
\eT


\bRemark{M1}
As already pointed out above, the \emph{existence} of the finite energy
weak solutions satisfying the energy inequality (\ref{m1}) was proved by
Paicu and Zarnescu \cite[Prop.~2]{PaiZar1} for certain potentials $F$.
\eR
\noindent%
\noindent%
The rest of the present paper is devoted to the proof of Theorem \ref{Tm1}.


\section{Decay for the $Q-$tensor}
\label{d}

We start by deriving decay estimates for solutions to the $Q-$tensor equation,
which we rewrite as
\bFormula{d1}
\partial_t \tn{Q} + \vu \cdot \Grad \tn{Q}  - \Del \tn{Q} = - \mathcal{L} [ \partial F (\tn{Q}) ] + \omega(\vu) \tn{Q} - \tn{Q} \omega (\vu),
\quad \tn{Q}(0, \cdot) = \tn{Q}_0.
\eF
The class of weak solutions considered in Theorem \ref{Tm1} has the $\tn{Q}-$component
in $L^1 \cap L^\infty(\RR^3)$ at least on compact time intervals,
therefore we may take the scalar product of (\ref{d1}) with
$2 G'(|\tn{Q} |^2) \tn{Q}$, where $G' \in C[0, \infty)$, and integrate over the physical space to obtain
\bFormula{d2}
 \frac{{\rm d}}{{\rm d}t} \intR{ G( |\tn{Q} |^2 ) }
  + \intR{ \left[ 2 G'(|\tn{Q}|^2) |\Grad \tn{Q} |^2
  + G''(|\tn{Q} |^2) \left| \Grad |\tn{Q}|^2 \right|^2 \right] }
 = - 2 \intR{  G' (|\tn{Q}|^2 ) \partial F (\tn{Q}) : \tn{Q} },
\eF
where we have used that
\bFormula{d3}
  [ \omega(\vu) \tn{Q} - \tn{Q} \omega (\vu) ]: \tn{Q} = 2 (\tn{Q} \tn{Q}) : \omega (\vu) = 0.
\eF
The relation (\ref{d2}) may be seen as a kind of \emph{renormalized}
energy balance for $\tn{Q}$. It is worth noting that, thanks to our
hypothesis $\xi = 0$, this relation is independent of the velocity $\vc{u}$.


\subsection{A maximum principle}
\label{MP}

Our first goal is to show that
\bFormula{d2a}
\| \tn{Q}(t,\cdot) \|_{L^\infty(\RR^3; \RR^{3 \times 3})} \leq r_2 \ \mbox{for all}\ t \geq 0
\eF
provided that the initial datum $\tn{Q}_0$ satisfies (\ref{m2}).
To this end, it is enough to take $G$ in (\ref{d2}) such that
\[
G(z) =0 \ \mbox{for}\ z \in [0,r_2^2], \ G' \geq 0, G'' \geq 0, \ G(z) > 0 \ \mbox{for}\ z > r_2^2.
\]
In view of the hypotheses (\ref{i7}), (\ref{m2}) we have
\[
\intR{ G(|\tn{Q}|^2)(t, \cdot) } \leq \intR{ G(|\tn{Q}_0 |^2) } = 0 \ \mbox{for all}\ t \geq 0
\]
yielding the desired conclusion (\ref{d2a}).

In what follows, in view of \eqref{d2a},  we may assume, by virtue of (\ref{i6}) and  (\ref{i7}), that
\bFormula{d3a}
0 \leq F(\tn{Q} ) \leq \alpha |\tn{Q}|^2, \ \alpha > 0.
\eF


\subsection{Asymptotic smallness of $\tn{Q}$}

As a consequence of (\ref{d2a}) and the energy inequality (\ref{m1}), we deduce that
\bFormula{d4}
  \sup_{t \geq 0} \left[ \| \tn{Q}(t, \cdot) \|_{L^\infty(\RR^3, \RR^{3 \times 3})}
   + \| F(\tn{Q}(t, \cdot) ) \|_{L^1(\RR^3)} +  \| \Grad \tn{Q} (t, \cdot) \|_{L^2(\RR^3;\RR^{3 \times 3 \times 3})} \right]
  \leq c.
\eF
In addition, there exists a sequence $t_n \to \infty$ such that
\bFormula{d5}
  \Del \tn{Q}(t_n, \cdot) - \mathcal{L} [ \partial F (\tn{Q} )](t_n, \cdot)
   = {g}_n \to 0 \ \mbox{in}\ L^2(\RR^3; \RR^{3 \times 3}) \ \mbox{as}\ n \to \infty.
\eF

Our goal is to show that (\ref{d4}), (\ref{d5}) imply that $\tn{Q}(t_n, \cdot)$ tends uniformly to zero, at least for a suitable subsequence of times.
 To this end, we need the following result that may be of independent interest.

\bLemma{p1}

Let $F \in C^2 ( B_{r_2} )$, $F(0) = 0$.
Suppose that $\tn{Q}$ is a solution of the stationary problem
\bFormula{pp1}
 - \Delta \tn{Q} + \mathcal{L} [ \partial F (\tn{Q})] = 0 \ \mbox{in}\ \RR^3
\eF
satisfying
\bFormula{pp2}
|\tn{Q}| \leq r_2 , \ |\Grad \tn{Q}|^2, \ F(\tn{Q}) \in L^1(\RR^3).
\eF
Then $\tn{Q}$ satisfies Pocho\v zaev's identity
\bFormula{pocho}
\int_{\RR^3} \left( \frac{1}{2} |\Grad \tn{Q}|^2 + 3 F(\tn{Q}) \right) \dx = 0.
\eF
In particular, $\tn{Q} \equiv 0$ provided that $F \geq 0$ in $B_{r_2}$.
\eL

\begin{proof}

We use the standard Pocho\v zaev type argument. To begin we claim that any solution of
(\ref{pp1}) is smooth (at least $C^2$) because of the standard
elliptic theory.

We multiply the equation on $\vc{x} \cdot \Grad \tn{Q}$ which is a symmetric traceless tensor. Accordingly
\[
0 =
-\Delta \tn{Q} : [\vc{x} \cdot \Grad \tn{Q}] + \partial F(\tn{Q}) : [\vc{x} \cdot \Grad \tn{Q}]
\]
\[
=
-\Div \left( ( \vc{x} \cdot \Grad \tn{Q}) : \Grad \tn{Q} \right) + |\Grad \tn{Q}|^2 + \frac{1}{2} \vc{x} \cdot \Grad |\Grad \tn{Q}|^2 + \Grad F(\tn{Q}) \cdot \vc{x}.
\]
Integrating the expression on the right-hand side over a ball $B_R \subset \RR^3$ of the radius $R$, we obtain
\bFormula{pp4}
 \int_{\partial B_R}  ( \vc{x} \cdot \Grad \tn{Q}) : ( \Grad \tn{Q} \cdot \vc{n} ) \ {\rm dS}_x
   - \int_{\partial B_R} \frac{1}{2}  |\Grad \tn{Q}|^2 \ \vc{x} \cdot \vc{n} \ {\rm dS}_x
   - \int_{\partial B_R} F(\tn{Q}) \vc{x} \cdot \vc{n}\ {\rm dS}_x
\eF
\[
 \mbox{} + \frac{1}{2} \int_{B_R} |\Grad \tn{Q}|^2 \ \dx
 + 3 \int_{B_R}  F(\tn{Q}) \ \dx  = 0.
\]
Since $\tn{Q}$ satisfies (\ref{pp2}) there exists a sequence $R_n \to \infty$ such that
\[
R_n \int_{\partial B_{R_n} } \left( |\Grad \tn{Q} |^2 + F(\tn{Q}) \right) \dx \to 0 \ \mbox{as}\ n \to \infty.
\]
Thus we may take $R = R_n$ in (\ref{pp4}) and let $n \to \infty$ to conclude that
\[
  \int_{\RR^3} \left( \frac{1}{2} |\Grad \tn{Q}|^2 + 3 F(\tn{Q}) \right) \dx = 0.
\]
\end{proof}
\noindent%
Going back to (\ref{d5}) we may assume, shifting $\tn{Q}(t_n, \cdot)$ in
$x$ as the case may be, that
\bFormula{d6}
  |\tn{Q}(t_n, 0)| \geq \frac{1}{2} \sup_{x \in \RR^3} | \tn{Q}(t_n, x) |.
\eF
Now, the relations (\ref{d4}), (\ref{d5}) imply that, at least for a suitable subsequence,
\[
  \tn{Q}(t_n, \cdot) \to \tilde { \tn{Q} } \ \mbox{in} \ C_{\rm loc}(\RR^3; \RR^{3 \times 3}),
\]
where $\tilde { \tn{Q} }$ is a solution of the stationary equation (\ref{pp1})
belonging to the class (\ref{pp2}), whence, by Lemma \ref{Lp1}, $\tilde {\tn{Q}} = 0$.

Thus, making use of (\ref{d6}), we obtain that
\bFormula{d7}
  \| \tn{Q}(t_n, \cdot) \|_{L^\infty(\RR^3; \RR^{3 \times 3})} \to 0 \ \mbox{as}\ t_n \to \infty,
\eF
at least for a suitable subsequence.

Finally, we may use the same arguments as in Section \ref{MP} to deduce from (\ref{d7})
and the hypothesis (\ref{i6}) the property
\bFormula{d8}
  |\tn{Q}(t, \cdot)| < r_1 \ \mbox{for all}\ t \ \mbox{large enough.}
\eF
More specifically, one could take in \eqref{d2} the function
\[
 G(z) = \left[ \left( z - \frac{r_1^2}4 \right)_+ \right]^2.
\]
Then, noting that
\[
 G'(|\tn{Q}|^2) \partial F(\tn{Q}) : \tn{Q} \ge - c G(|\tn{Q}|^2)
  \ \text{for}\ |\tn{Q}| \in [r_1/2,r_2],
\]
in view of \eqref{d7}
we can choose $\tilde{t}$ such that $ \| \tn{Q}(\tilde{t},\cdot) \|_{L^\infty} \le r_1/2 $
and apply Gronwall's lemma starting from the time~$\tilde{t}$.
Hence, in view of \eqref{i7}, we may assume that $F$,
in addition to (\ref{d3a}), satisfies
\bFormula{d9}
  \partial F (\tn{Q}) : \tn{Q} \geq 0.
\eF
Thus, revisiting (\ref{d7}), we may infer that
\bFormula{d8a}
  \| \tn{Q} (t, \cdot) \|_{L^\infty(\RR^3, \RR^{3 \times 3})} \to 0 \ \mbox{as}\ t \to \infty.
\eF


\subsection{$L^2-$decay of $\tn{Q}$}

Using an approximation by smooth functions, we can take
\[
  G(z) = \sqrt{z}
\]
in (\ref{d2}).
It is worth noting that the above function $G$ is not convex; nevertheless,
by a direct computation one can check that the sum
$2 G'(|\tn{Q}|^2) |\Grad \tn{Q} |^2 + G''(|\tn{Q} |^2) \left| \Grad |\tn{Q}|^2 \right|^2$
is nonnegative anyway.
Hence, by virtue of (\ref{d8}), (\ref{d9}), we can conclude that
\bFormula{d10}
  \| \tn{Q} (t, \cdot) \|_{L^1(\RR^3; \RR^{3 \times 3})} \leq c \ \mbox{for all}\ t > 0.
\eF
To be more precise, we have to notice that
\eqref{d8} has been justified so far only for $t$ greater than some
(sufficiently large) time $\tilde T$. Hence, relation \eqref{d10} should be
proved first on the time interval $[0,\tilde T]$ by integrating~\eqref{d2}
and using the Gronwall lemma (indeed, on $[0,\tilde T]$, the right hand side
of \eqref{d2} needs not be negative),
and subsequently extended for $t\ge \tilde T$ by means of~\eqref{d8}.
Finally, taking $G(z) = z$ in (\ref{d2}) we obtain
\bFormula{d11}
  \frac{{\rm d}}{{\rm d}t} \intR{ \frac{1}{2} |\tn{Q}|^2 }
   + \intR{ |\Grad \tn{Q} |^2 } \leq 0 \ \mbox{for all}\ t > 0 \ \mbox{large enough.}
\eF
Recalling (\ref{d8a}) and
using the standard interpolation inequalities, we get
\[
  \| \tn{Q} \|_{L^2(\RR^3; \RR^{3 \times 3})}
    \leq \| \tn{Q} \|_{L^1(\RR^3; \RR^{3 \times 3})}^{2/5} \| \tn{Q} \|_{L^6(\RR^3; \RR^{3 \times 3})}^{3/5}
    \leq c \| \Grad \tn{Q} \|^{3/5}_{L^2(\RR^3, \RR^{27})},
\]
whence (\ref{d11}) implies
\bFormula{d12}
  \| \tn{Q} (t, \cdot) \|_{L^2(\RR^3; \RR^{3\times 3})}^2
    \leq c (1 + t)^{-3/2} \ \mbox{for all} \ t \geq 0.
\eF

If $F$ satisfies the hypothesis (\ref{m3bis}) the $L^2-$decay rate is exponential, specifically
\bFormula{d12bis}
  \| \tn{Q} (t, \cdot) \|_{L^2(\RR^3; \RR^{3\times 3})}^2
    \leq c \exp(-d t) \ \mbox{for all} \ t \geq 0 \ \mbox{and some}\ d > 0.
\eF
Indeed, since $F$ is twice continuously differentiable, (\ref{m3bis}) implies strict positivity of the Hessian of $F$ at $\tn{Q} = 0$, in particular,
$F$ is strictly convex in a neighborhood of zero. Consequently,
\[
\partial F(\tn{Q}) : \tn{Q} \geq F(\tn{Q}) \geq \lambda |\tn{Q}|^2
\]
and (\ref{d12bis}) follows from (\ref{d2}) with $G(z) = z$.


\section{Decay for the Navier-Stokes system via the Fourier splitting method}
\label{v}

Using elementary inequalities, the (differential version of the)
energy inequality~(\ref{m1}) can be rewritten in the form
\bFormula{v1-0}
  \frac{{\rm d}}{{\rm d}t} \intR{ \left[ \frac{1}{2} |\vu|^2
      + \frac{1}{2} |\Grad \tn{Q} |^2 + F(\tn{Q}) \right] }
   + \intR{ \left[ |\Grad \vu |^2 + \frac12 | \Del \tn{Q} |^2
       - c \left| \mathcal{L}[\partial F(\tn{Q})] \right|^2 \right] }
   \le 0.
\eF
Moving the last integrand to the right hand side,
and noting that, thanks to (\ref{d3a}),
\bFormula{v1x}
  \intR{ \left| \mathcal{L}[\partial F(\tn{Q})] \right|^2 }
  \le c \intR { | \tn{Q} |^2 },
\eF
%
%
we then obtain
%
%
%
%
%
%
\bFormula{v1}
  \frac{{\rm d}}{{\rm d}t} \intR{ \left[ \frac{1}{2} |\vu|^2
   + \frac{1}{2} |\Grad \tn{Q} |^2 + F(\tn{Q}) \right] (t, \cdot) }
  + \intR{ \left[ |\Grad \vu |^2 + \frac12 | \Del \tn{Q} |^2  \right] } \leq c (1 + t)^{-3/2},
\eF
thanks also to~(\ref{d12}).

The extra term on the right-hand side of (\ref{v1}) is responsible for the loss
of the exponential decay rate for $\tn{Q}$ in the general case. Actually, if $F$ satisfies~\eqref{m3bis},
then the energy inequality reads
\bFormula{v1bis}
  \frac{{\rm d}}{{\rm d}t} \intR{ \left[ \frac{1}{2} |\vu|^2
   + \frac{1}{2} |\Grad \tn{Q} |^2 + F(\tn{Q}) \right] (t, \cdot) }
   + c \intR{ \left[ \frac{1}{2} |\Grad \vu|^2
   + \frac{1}{2} |\Grad \tn{Q} |^2 + F(\tn{Q}) \right] (t, \cdot) } \leq 0, \ c > 0.
\eF
Indeed we have
\[
 \int_{\RR^3} \left| \Delta \tn{Q} - \mathcal{L} \left[ \partial F(\tn{Q}) \right] \right|^2 \ \dx
 \geq c \int_{\RR^3} \left[ \frac{1}{2} | \Grad \tn{Q}|^2 + F(\tn{Q}) \right] \ \dx, \ c > 0,
\]
for any $\tn{Q}$ in an open neighborhood of zero. To see this, denote
\[
- \Delta \tn{Q} + \mathcal{L} \left[ \partial F (\tn{Q} ) \right] = \tn{G}
\]
and observe that
\[
\int_{\RR^3} \left[ \frac{1}{2} |\Grad \tn{Q} |^2 + F(\tn{Q}) \right] \ \dx \leq
c
\int_{\RR^3} \left[ |\Grad \tn{Q}|^2 + \partial F(\tn{Q}) : \tn{Q} \right] \ \dx
 = c \int_{\RR^3} \tn{G} : \tn{Q} \ \dx;
\]
where the term on the right-hand side may be ``absorbed'' by means of the Cauchy-Schwartz
inequality provided that $F$ satisfies (\ref{m3bis}).


\subsection{Fourier analysis for the Navier-Stokes system}

Let
\[
\widehat{v} (t, \xi) = \frac{1}{(2\pi)^{3/2}} \intR{ \exp(- {\rm i} \xi \cdot x ) v(t,x) }
\]
denote the Fourier transform of a function $v$ with respect to the
spatial variable $x$. Accordingly, the velocity field $\vu$, solving
(\ref{p7}), can be written as
\bFormula{v2}
\widehat{u}_i(t, \xi) = \exp \left( - |\xi|^2 t \right) \widehat{u}_{0,i}(\xi)
\eF
\[
+ \int_0^t \exp \left( - |\xi|^2 (t - s) \right)  \left[ \left( \delta_{i,j} - \frac{\xi_i \xi_j}{|\xi|^2} \right) \xi_k
 \left( - \widehat{\tn{Q} \Del \tn{Q}}  + \widehat{\Del \tn{Q} \tn{Q}}
  - \widehat{\Grad \tn{Q} \odot \Grad \tn{Q}}
  - \widehat{\vu \otimes \vu} \right)_{j,k} (s,\xi) \right] \ {\rm d}s.
\]
Now, we observe that
\[
  \| ( \Grad \tn{Q} \odot \Grad \tn{Q})(t, \cdot) \|_{L^1(\RR^3; \RR^3)}
   + \| (\vu \otimes \vu) (t, \cdot) \|_{L^1(\RR^3;\RR^3)} \leq c \mathcal{E}(t) \ \mbox{for all}\ t \geq 0,
\]
where
\[
\mathcal{E}(t) = \int_{\RR^3} \left( \frac{1}{2} |\vu|^2 + \frac{1}{2} | \Grad \tn{Q} |^2 + F(\tn{Q}) \right) \ \dx.
\]
Consequently,
\bFormula{v3}
  \left| \widehat{\Grad \tn{Q} \odot \Grad \tn{Q}} (t, \xi) \right|
    +  \left| \widehat{\vu \otimes \vu} (t, \xi) \right| \leq c \mathcal{E}(t) \ \mbox{for all} \ t, \xi.
\eF
Next, writing
\[
 (\tn{Q} \Del \tn{Q})_{ij}
  = \tn{Q}_{ik} ( \partial_l \partial_l \tn{Q}_{kj} )
  = \partial_l \left( \tn{Q}_{ik} \partial_l \tn{Q}_{kj} \right)
      - \partial_l \tn{Q}_{ik} \partial_l \tn{Q}_{kj}
\]
and, analogously,
\[
 (\Del \tn{Q} \tn{Q})_{ij}
  = (\partial_l \partial_l \tn{Q}_{ik}) \tn{Q}_{kj}
  = \partial_l \left( \partial_l \tn{Q}_{ik} \tn{Q}_{kj} \right)
      - \partial_l \tn{Q}_{ik} \partial_l \tn{Q}_{kj},
\]
we may therefore infer that
\bFormula{v4}
  \left|-\widehat{ \tn{Q} \Del \tn{Q}} (t, \xi) + \widehat{ \Del \tn{Q} \tn{Q} } (t, \xi) \right|
  = \left| \widehat{ \Div ( \Grad \tn{Q} \tn{Q} ) } - \widehat{ \Div ( \tn{Q} \Grad \tn{Q} ) } \right|
\eF
\[
  \leq c(1 + |\xi| ) \left( \mathcal{E}(t) + \int_{\RR^3} |\tn{Q} |^2(t, \cdot) \ \dx \right)
  \leq  c(1 + |\xi| ) \left( \mathcal{E}(t) + (1 + t)^{-3/2} \right) \
  \mbox{for all}\ t, \xi.
\]
Thus, combining (\ref{v3}), (\ref{v4}) with (\ref{v2}) we conclude
\bFormula{v5}
|\widehat{u}_i|(t, \xi) \leq \exp \left( - |\xi|^2 t \right) |\widehat{u}_{0,i}|(\xi)
\eF
\[
 \mbox{}
  + c \int_0^t \exp \left( - |\xi|^2 (t - s) \right) |\xi| ( 1 + |\xi| )\left( \mathcal{E}(s) + (1 + s)^{-3/2} \right) {\rm d}s,
   \ i=1,2,3.
\]
Here again, we remark that (\ref{v5}) does not contain the extra term $(1 + s)^{-3/2}$ if $F$ satisfies (\ref{m3bis}).


\subsection{First decay estimate}

Having collected all the necessary ingredients, we are ready to finish the proof of
Theorem~\ref{Tm1}. We focus on the case of a general nonlinearity $F$ and then
shortly comment on how to modify the arguments when $F$ satisfies (\ref{m3bis}).
To begin, note that
\bFormula{kk}
\intR{ \left[ \frac12 | \Grad \tn{Q} |^2 + F(\tn{Q}) \right] }
  \le \intR { \frac12 | \Del \tn{Q} |^2 + c_1 \left[ F(\tn{Q}) + \frac{1}{2} | \tn{Q} |^2 \right] }
\eF
\[
  \le \intR { \left[ \frac12 | \Del \tn{Q} |^2 + c_2 | \tn{Q} |^2 \right] }
  \le \intR { \frac12 | \Del \tn{Q} |^2 } + c_3 (1+t)^{-3/2}.
\]
Adding \eqref{kk} to \eqref{v1} and applying Plancherel's Theorem,
we obtain
\bFormula{EI}
\begin{split}
\frac{{\rm d}}{{\rm d}t} \intR{ \left[ \frac{1}{2} |\vu|^2
    + \frac{1}{2} | \Grad {\tn{Q}} |^2 + F(\tn{Q})  \right] (t, \cdot) }
  + \int_{\mathbb R^3} \Big| |\xi| \widehat{ \vu } \Big|^2 \ {\rm d}\xi \\
  \mbox{} +  \intR{ \left[\frac12 | \Grad \tn{Q}  |^2  + F(\tn{Q})  \right] } \leq c (1 + t)^{-3/2}.
\end{split}
\eF
Next, we have
\bFormula{v7}
\int_{\RR^3} \Big| |\xi| \widehat{\vu} \Big|^2 \ {\rm d}\xi
= \int_{|\xi| < R(t)} \Big| |\xi| \widehat{\vu} \Big|^2 \ {\rm d}\xi + \int_{|\xi| \geq R(t)} \Big| |\xi| \widehat{\vu} \Big|^2 \ {\rm d}\xi
\geq R^2 (t) \int_{|\xi| \geq R(t)} | \widehat{\vu} |^2 \ {\rm d}\xi.
\eF
Replacing \eqref{v7} into \eqref{EI},  we infer
\bFormula{v7b}
 \frac{{\rm d}}{{\rm d}t} \mathcal{E}(t)
   + R^2(t) \mathcal{E}(t)
  \leq R^2(t) \int_{|\xi|<R(t)} \big| \widehat{ \vu } \big|^2 \ {\rm d}\xi
   + c (1 + t)^{-3/2} \ \mbox{for any}\ 0 \leq R(t) \leq 1.
\eF
In order to evaluate the integral on the right hand side, we notice that,
in agreement with (\ref{v5}),
\bFormula{v8}
 \int_{|\xi| < R(t)} |\widehat{\vu}|^2 \ {\rm d}\xi
\eF

\[
 \leq c \| \vu_0 \|^2_{L^1(\RR^3; \RR^3)} \int_{|\xi| < R(t)} \exp \left( - 2 |\xi|^2 t \right) \ {\rm d}\xi
\]
\[
 \mbox{} + c \int_{|\xi| < R(t)} \left( \int_0^t \exp \left( - |\xi|^2 (t - s) \right)
       |\xi| ( 1 + |\xi| ) \left( \mathcal{E}(s) + (1 + s)^{-3/2} \right)   {\rm d}s \right)^2 \ {\rm d}\xi.
\]
Now, let us evaluate the first integral on the right hand side: passing to polar
coordinates and then substituting $r:=\rho^3(t+1)^{3/2}$ we get
\bFormula{v8b}
 \int_{|\xi| < R(t)} \exp \left( - 2 |\xi|^2 t \right) \ {\rm d}\xi
  \le c \int_0^{R(t)} \exp \left( - 2 \rho^2 t \right) \rho^2 \ {\rm d}\rho
  \le c \int_0^{R(t)} \exp \left( - 2 \rho^2 ( t + 1 ) \right) \rho^2 \ {\rm d}\rho
\eF
\[
  = c (t+1)^{-3/2} \int_0^{R^3(t)(t+1)^{3/2}} e^{-2 r^{2/3} } \ {\rm d}r
  \le c (t+1)^{-3/2} \int_0^{+\infty} e^{-2 r^{2/3} } \ {\rm d}r
  \le c (t+1)^{-3/2},
\]
where we also assumed (and used) the fact that $R(t)$
will be chosen to be smaller than $1$.

Collecting (\ref{v7b} - \ref{v8b}), we then conclude that
\bFormula{vvv}
  \frac{{\rm d}}{{\rm d}t} \mathcal{E}(t)
   + R^2(t) \mathcal{E}(t)
\eF
\[
 \leq c \left[  (t+1)^{-3/2}
   + R^2(t) \int_{|\xi| < R(t)} \left( \int_0^t \exp \left( - |\xi|^2 (t - s) \right)
       |\xi| \left( \mathcal{E}(s) + (1 + s)^{-3/2} \right)   {\rm d}s \right)^2 \ {\rm d}\xi  \right]
\]
\[
= c \left[  (t+1)^{-3/2} + R^2(t) \int_0^{R(t)} \left( \int_0^t \exp \left( - r^2 (t - s) \right)
       r^2 \left( \mathcal{E}(s) + (1 + s)^{-3/2} \right)   {\rm d}s \right)^2 \ {\rm d}r  \right]
\]
for any  $0 \leq R(t) \leq 1$.


\subsection{A bootstrap argument}

The inequality (\ref{vvv}) is a starting point of a bootstrap procedure to deduce
the desired decay estimate (\ref{m3}). We start with
an auxiliary assertion.
\begin{Lemma}\label{l:var}
 Let $\gamma \in (0,1)$, $\mu >0$ and $\gamma <\mu$. If
 \bFormula{diff}
 \frac{{\rm d}}{{\rm d}t} \mathcal{E}(t) +  (1 + t)^{-\gamma} \mathcal{E}(t) \leq c (1 + t)^{-\mu},
 \eF
 then
 \bFormula{decay}
   \mathcal{E} (t) \le c(\gamma,\E(0)) (1 + t)^{-\mu+\gamma}.
 \eF
\end{Lemma}
\begin{proof}
Let $g(t) = \exp\left[ (1-\gamma)^{-1}(1+t)^{1-\gamma}\right]$. Then (\ref{diff}) yields
\[
  \frac{{\rm d}}{{\rm d}t} (g(t) \mathcal{E}(t) ) \leq  c g(t) (1 + t)^{-\mu}.
\]
Hence,
\begin{equation} \label{1}
  \mathcal{E}(t)  \leq g(t)^{-1} \mathcal{E}(0)
   + c g(t)^{-1} \int_0^t (1+s)^{-\mu} g(s) \ \ds =: I_1(t) + I_{2}(t).
\end{equation}
Clearly it is sufficient to handle $I_{2}$. Noting that
$g(t)=(1+t)^\gamma g'(t)$ and then integrating by
parts, we obtain
\[
  I_2(t)  = c g(t)^{-1} \int_0^t (1+s)^{ \gamma - \mu} g'(s) \ \ds
   \le c (1+t)^{\gamma-\mu}
    + c (\mu-\gamma) g(t)^{-1} \int_0^t g(s) (1+s)^{ \gamma - \mu -1 }  \ \ds.
\]
To control the last term we simply split it as
\[
  c (\mu-\gamma) g(t)^{-1} \int_0^t g(s) (1+s)^{ \gamma - \mu -1 }  \ \ds
  = c (\mu-\gamma) g(t)^{-1} \int_0^{t/2} g(s) (1+s)^{ \gamma - \mu -1 }  \ \ds
\]
\[
  \mbox{}
   + c (\mu-\gamma) g(t)^{-1} \int_{t/2}^t g(s) (1+s)^{ \gamma - \mu -1 }  \ \ds
  =: I_3 + I_4
\]
where
\[
I_3\leq c (\mu-\gamma)  \frac{g(t/2)}{g(t)} \int_0^{t/2} (1+s)^{ \gamma - \mu -1 } \ \ds
\leq c \frac{g(t/2)}{g(t)}
\]
which decays exponentially fast, and
\[
I_4\leq c (\mu-\gamma) (1+\frac t2)^{ \gamma - \mu -1 } \int_{t/2}^t \frac{g(s)}{g(t)}   \ \ds
\leq c (\mu-\gamma) (1+t)^{ \gamma - \mu}.
\]
The lemma is proved.
\end{proof}
\noindent
Now, we are ready to start bootstraping (\ref{vvv}). Suppose we have already shown
\bFormula{boot}
  \mathcal{E}(t) \leq c {(1 + t)^{-\alpha}}, \ 0 \leq \alpha.
\eF
Accordingly, the inequality (\ref{vvv}) gives rise to
\bFormula{vvv1}
\frac{{\rm d}}{{\rm d}t} \mathcal{E}(t)
   + R^2(t) \mathcal{E}(t)
\eF
\[
  \leq c \left[  (t+1)^{-3/2} + R^2(t) \int_0^{R(t)} \left( \int_0^t \exp \left( - r^2 (t - s) \right)
       r^2 \frac{1}{(1 + s)^\alpha} \,  {\rm d}s \right)^2 \ {\rm d}r  \right],
\]
where
\[
\int_0^t \exp \left( - r^2 (t - s) \right)
       r^2 \frac{1}{(1 + s)^\alpha} \, {\rm d}s = \exp \left( - r^2 t \right) r^{2 \alpha} \int_0^t \exp \left(  r^2 s \right)
       r^2 \frac{1}{(r^2 + r^2 s)^\alpha} \,  {\rm d}s
\]
\[
= r^{2 \alpha} \exp \left( - r^2 t \right)  \int_0^{r^2 t} \frac{ \exp z }
       {(r^2 + z)^\alpha} \,  {\rm d}z.
\]
Let us now observe that
\begin{equation}\label{vvv2}
  \displaystyle
  \exp \left( - r^2 t \right)  \int_0^{r^2 t} \frac{ \exp z }
       {(r^2 + z)^\alpha}   {\rm d}z \leq
       \begin{cases}
         1, & \;\mbox{if}\; \alpha =0,\\
        \frac{1}{1-\alpha} [( r^2(1+t) )^{1-\alpha}-r^{2(1-\alpha)}] \leq  \frac{ r^{2(1-\alpha)} }{1-\alpha}, &\;\mbox{if}\;  0<\alpha<1,\\
       \frac{1}{1-\alpha} [( r^2(1+t) )^{1-\alpha}-r^{2(1-\alpha)})]\leq  \frac{ r^{2(1-\alpha)} }{\alpha-1},   &\;\mbox{if}\;  1<\alpha.
  \end{cases}
 \end{equation}
Hence
\begin{equation}
r^{2 \alpha} \exp \left( - r^2 t \right)  \int_0^{r^2 t} \frac{ \exp z }
       {(r^2 + z)^\alpha} \,  {\rm d}z \leq
       \begin{cases}
   1, & \;\mbox{if}\; \alpha =0,\\
c(\alpha) r^2, &  \;\mbox{if} \;\alpha>0,\, \mbox{and}\;\alpha \neq 1.
\end{cases}
\end{equation}
Thus, \eqref{vvv1} gives
\bFormula{vvv3b}
  \frac{{\rm d}}{{\rm d}t} \mathcal{E}(t)
   + R^2(t) \mathcal{E}(t) \leq
   \begin{cases} c(\alpha) \left[  (t+1)^{-3/2} + R^3 (t) \right], & \;\mbox{if}\; \alpha =0,\\
     c(\alpha) \left[  (t+1)^{-3/2} + R^{7} (t) \right], &  \;\mbox{if} \;\alpha>0,\, \mbox{and}\;\alpha \neq 1.\end{cases}
\eF
Now, from the uniform boundedness of the energy, we know that
\eqref{boot} holds for $\alpha =0$.
Hence, taking $R(t) =(1+t)^{-\beta}$ with
\[
  \beta = \frac{1}{2} -\frac{\epsilon}{3},
\]
where $\epsilon>0$ is a small number, we obtain
$2\beta= 1- \frac{2\epsilon}{3}<1$ and $R^3 (t) = (1+t)^{-\frac{3}{2}+\epsilon}$
yielding
\[
\frac{{\rm d}}{{\rm d}t} \mathcal{E}(t)
   + (1 + t)^{-2 \beta} \mathcal{E}(t) \leq c(\alpha) (t+1)^{-3/2+\epsilon }.
\]
Since $2\beta <1$ (this is  why we subtracted  $\epsilon > 0$),
Lemma~\ref{l:var} can be applied with $\gamma =2\beta$ and  $ \mu =3/2-\epsilon$,
whence
\[
  \E(t) \leq c(t+1)^{- \mu + 2\beta} = c(t+1)^{ -\frac{1}{2} +\frac{\epsilon}{3}}.
\]
Now we can take $2\beta = \frac{3}{7}$. Repeating the above argument
with $\alpha = \frac{1}{2} - \frac{\epsilon}{3}$,
and referring to the second row of formula (\ref{vvv3b}) yields
\bFormula{vvv3d}
  \E(t) \leq c(t+1)^{- \frac{3}{2} + \frac{3}{7} }= c(t+1)^{ -\frac{15}{14}}.
\eF
We are now ready to provide a more refined estimate of the quantity on the
right hand side of \eqref{vvv}. Using \eqref{vvv3d},
and noting that $\frac{3}{2} > \frac{15}{14}$, we only need to handle the term
depending on $\mathcal{E}$ in \eqref{vvv}. Actually,
by \eqref{vvv3d} we have
\[
   \int_0^t \exp \left( - r^2 (t - s) \right)
       r^2 \mathcal{E}(s) \, {\rm d}s  \leq    \int_0^t \exp \left( - r^2 (t - s) \right)
       r^2 (s+1)^{-\frac{15}{14}} \,{\rm d}s.
\]
Squaring and splitting the integral we obtain
\[
  \left(\int_0^t \exp \left( - r^2 (t - s)\right )
       r^2 (s+1)^{-\frac{15}{14}} \,{\rm d}s \right)^2
     \le  2\left( \int_0^{\frac{t}{2}} \exp \left( - r^2 (t - s)\right )
       r^2 (s+1)^{-\frac{15}{14}} \,{\rm d}s\right)^2 +
\]
\[
   2 \left( \int_{\frac{t}{2}}^t \exp \left( - r^2 (t - s)\right )
       r^2 (s+1)^{-\frac{15}{14}} \, {\rm d}s \right)^2 =: 2J_1 + 2J_2.
\]
Then, Jensen's inequality  yields
\[
   J_1 \leq \frac{t}2 \left(\int_0^{t/2}\exp \left( -2 r^2 (t - s) \right)
       r^4 (s+1)^{-\frac{30}{14}}    \,{\rm d}s \right) =: \frac{t}2 J_3.
\]
Thus $\Aa := R^2(t) \int_0^{R(t)} \left( \int_0^t \exp \left( - r^2 (t - s) \right) r^2
\left( \mathcal{E}(s) + (1 + s)^{-3/2} \right)   \, {\rm d}s \right)^2  \, {\rm d}r $
can be bounded as follows:
\bg \label{a}
  \Aa\leq c R^2(t) \left[t\int_0^{R(t)}  J_3  \,{\rm d}r
    + \int_0^{R(t)} J_2 \, {\rm d}r\right]
     =: c R^2(t) [\Aa_1 +\Aa_2 ].
\ed

\smallskip

To control $\Aa_1$,  change the order of integration:
\[
  \Aa_1 = t \int_0^{\frac{t}{2}} \int_0^{R(t)} \exp \left( - 2 r^2 (t - s)\right )
       r^4 (s+1)^{-\frac{30}{14}}\,{\rm d}r\, {\rm d}s,
\]
and  make the change of variables $\sigma = \sqrt 2r\sqrt{t-s}$. Since  for $s\in (0,t/2)$
we have $(t-s)^{-5/2} \leq c t^{-5/2}$, we get
\bg \label{a1}
  \Aa_1 =  \frac{t}{2^{5/2}}
  \int_0^{\frac{t}{2}}\frac{1}{(s+1)^{30/14} } \frac{1}{(t-s)^{5/2}}\int_0^{\sqrt{2} R(t) \sqrt{t-s}}\sigma^4 \exp(-\sigma^2)  \,
   {\rm d}\sigma \, {\rm d}s
\ed
\[
  \leq   Ct^{-3/2}
  \int_0^{\infty}\sigma^4\exp(-\sigma^2) \, {\rm d}\sigma
 \leq Ct^{-3/2}.
\]

\smallskip

Next, let us estimate $\Aa_2$. To this aim, note that for $s \in (t/2,t) $ we have \\ $ (s+1)^{-15/14} \leq c (t+1)^{-15/14}$,
and   since
\[ \int_{t/2}^t \exp \left( - r^2( t-s) \right) r^2 \,{\rm d}s
 = 1- \exp ( - r^2( t/2) )\leq 1,
\]
the term $\Aa_2$ can be bounded by
\bg \label{a2}
\begin{split}
 \Aa_2 &=  \int_0^{R(t)}\left(\int_{t/2}^t \exp \left( - r^2( t-s)\right ) r^2  {(s+1)^{-15/14} } \, {\rm d}s\right)^2 \, {\rm d}r \\
 &\leq C(t+1)^{-30/14}  \int_0^{R(t)} \left(\int_{t/2}^t\exp \left( - r^2( t-s)\right ) r^2\, {\rm d}s \right)^2  \,{\rm d}r.
\end{split}
\ed
Thus
\bg \label{a3}  \Aa_2  \leq CR(t)\, (t+1)^{-30/14}. \ed
Combining inequalities \eqref{vvv}, \eqref{a}, \eqref{a1}, \eqref{a2} and \eqref{a3},
since $\frac{30}{14} >\frac{3}{2}$, yields
\bg \label{final}
\begin{split}
& \frac{{\rm d}}{{\rm d}t} \mathcal{E}(t)
   + R^2(t) \mathcal{E}(t) \\
   & \leq C\left( R^2(t) t^{-3/2} +(t+1)^{-3/2}  + R^3(t) (t+1)^{-30/14}\right)\\
   & \leq C\left( R^2(t) t^{-3/2} +(t+1)^{-3/2}  + R^3(t) (t+1)^{-3/2}\right).
\end{split}
\ed
Choose $R(t)=1$  and since $ (t+1)^{-3/2}\leq t^{-3/2}  $
\bg \label{final1}
 \frac{{\rm d}}{{\rm d}t} \mathcal{E}(t)
   +  \mathcal{E}(t)
   \leq Ct^{-3/2}.
\ed
Variation of parameters,  with the multiplier $e^{-t}$,  yields integrating  over $ [1,t]$
\bg \label{f}
  \mathcal{E}(t) \leq \mathcal{E}(1)e^{-(t-1)}  + C \int_1^t e^{-(t-s)} s^{-3/2} \, {\rm d}s.
\ed
Splitting the integral on the right hand side gives
\bg \label{help}\begin{split}
    \int_1^t e^{-(t-s)} s^{-3/2} \, {\rm d}s
     =\int_1^{t/2}e^{-(t-s)} s^{-3/2} \, {\rm d}s +\int_{t/2}^t e^{-(t-s)} s^{-3/2} \,{\rm d}s \leq\\
   2e^{-t/2} ( 1 - t^{-1/2} ) + (t/2)^{-3/2} (1-e^{-t/2} )\leq C (t/2)^{-3/2}.
  \end{split}
\ed
Combining \eqref{f} and \eqref{help} yields the desired decay rate, which is optimal since it coincides
with the underlying linear part:
\[ \mathcal{E}(t) \leq C (t/2)^{-3/2}\]
This concludes the proof of Theorem~\ref{Tm1}.


\section{Examples}
\label{D}

We consider the class of polynomial potentials investigated by Paicu and Zarnescu \cite{PaiZar2}, specifically,
\bFormula{D1}
F(\tn{Q}) = \frac{a}{2} |\tn{Q}|^2 + \frac{b}{3} \trace[\tn{Q}^3 ] + \frac{c}{4} |\tn{Q}|^4,
\eF
observing that
\[
\mathcal{L}[\partial F (\tn{Q})] = a \tn{Q}
+ b \left( \tn{Q}^2 - \frac{1}{3} \trace[ \tn{Q}^2 ] \tn{I} \right) + c |\tn{Q}|^2 \tn{Q}.
\]
Here and hereafter, $a,b,c$ are real parameters.


\subsection{The case $a > 0$}

If $a > 0$, the isotropic state is at least locally stable. It is easy to check that
there exists an open neighborhood $\mathcal{O}$ of $0$ such that
$F$ satisfies the hypotheses of Theorem \ref{Tm1} including (\ref{m3bis}). Accordingly, we get
\bFormula{D2}
   \| \vu(t, \cdot) \|_{L^2(\RR^3;\RR^3)}
     + \| \tn{Q}(t, \cdot) \|_{W^{1,2}(\RR^3; \RR^{3 \times 3})}  \leq c (1 + t)^{- 3/4 }
\eF
provided that there exists $t_0$ such that $\tn{Q}(t_0, \cdot) \in \mathcal{O}$
a.e.~in~$\RR^3$. The decay rate is global (unconditional) if $c > 0$
and $|b| \leq \overline{b}(a,c)$.


\subsection{The case $a \leq 0$, $c > 0$}

We consider solutions belonging to the class
\bFormula{D3}
 \tn{Q}(t, \cdot) \in D^{1,2}(\RR^3; \RR^{3 \times 3}), \ F(\tn{Q}) \in L^1(\RR^3),
\eF
where $D^{1,2}$ denotes the completion of $\DC$ under the $\| \Grad \tn{Q} \|_{L^2}$ norm.

To begin, we observe that $\tilde{\tn{Q}} \equiv 0$ is the only stationary
solution belonging to the class (\ref{D3}). Indeed Pocho\v zaev's identity
(\ref{pocho}) reads
\[
  \int_{\RR^3} \left( \frac{1}{2} |\Grad \tilde{\tn{Q}}|^2 + \frac{3a}{2} |\tilde{\tn{Q}}|^2
   + {b} \trace[\tilde{\tn{Q}}^3 ] + \frac{3c}{4} |\tilde{\tn{Q}}|^4\right) \dx = 0,
\]
while, as $\tilde{\tn{Q}}$ solves the stationary problem,
\[
  \int_{\RR^3} \left(  |\Grad \tilde{\tn{Q}}|^2 + a |\tilde{\tn{Q}}|^2
   + b \trace[ \tilde{\tn{Q}}^3]  + c |\tilde {\tn{Q}}|^4 \right) \dx = 0.
\]
Consequently,
\[
  \int_{\RR^3} \left( -\frac{1}{2} |\Grad \tilde{\tn{Q}}|^2 + \frac{a}{2} |\tilde{\tn{Q}}|^2
    - \frac{c}{4} |\tilde {\tn{Q}}|^4 \right) \dx = 0,
\]
yielding the desired conclusion $\tilde{\tn{Q}} = 0$.

We claim the following:


{\it Suppose that
\bFormula{D4}
 \| \tn{Q}(t, \cdot) \|_{L^2(\RR^3; \RR^{3 \times 3})} \leq M \ \mbox{for a.a.}\ t > 0.
\eF
Then
\bFormula{D5}
 \tn{Q}(t, \cdot) \to 0 \ \mbox{in}\ L^\infty(\RR^3; \RR^{3 \times 3}) \ \mbox{as}\ t \to \infty.
\eF
}

In order to see (\ref{D5}) we first observe that (\ref{D4}) implies that the energy
\[
  \mathcal{E}(t) = \intR{ \left[ \frac{1}{2} |\vu|^2 + \frac{1}{2} |\Grad \tn{Q} |^2 + F(\tn{Q}) \right] (t, \cdot) }
\]
remains bounded, more specifically, $\mathcal{E}$ tends to a finite limit as $t \to \infty$.

This in turn implies the existence of a sequence $t_n \in [n, n+ 1]$ such that
\[
  - \Delta \tn{Q}(t_n, \cdot) + \mathcal{L} [ \partial F(\tn{Q}(t_n , \cdot) ] \to 0
   \ \mbox{in}\ L^2(\RR^3; \RR^{3 \times 3}) \ \mbox{as}\ t_n \to \infty.
\]
However, (\ref{D4}), together with the standard elliptic regularity
estimates, implies that $\tn{Q}(t_n, \cdot)$ converges uniformly to
a stationary solution, meaning to zero. We may therefore infer that
\[
\|\tn{Q}(t_n, \cdot)\|_{L^\infty(\RR^3; \RR^{3 \times 3})}=q_n\to 0 \quad n\to \infty.
\]
Going back to (\ref{d2}) and taking $G(z)=z^{p/2}$ we deduce
\[
  \| \tn{Q}(t, \cdot) \|_{L^p(\RR^3; \RR^{3 \times 3})}
   \leq \exp{\left(K (t-s)\right)} \| \tn{Q}(s, \cdot) \|_{L^p(\RR^3; \RR^{3 \times 3})},\ p> 2,  \ t \geq s,
\]
where $K$ does not depend on $p$.
In particular, by interpolation, we get
\begin{align}\nonumber
\|\tn{Q}(t, \cdot)\|_{L^\infty(\RR^3; \RR^{3 \times 3})}&
%
=\sup_{x\in\RR^3}\|\tn{Q}(t, \cdot)\|_{L^\infty(B(x,1);\RR^{3\times 3})}
= \sup_{x\in\RR^3} \limsup_{p\to\infty}\|\tn{Q}(t,\cdot)\|_{L^p(B(x,1); \RR^{3 \times 3})}\\
\nonumber
&\leq \limsup_{p\to\infty}\|\tn{Q}(t,\cdot)\|_{L^p(\RR^3; \RR^{3 \times 3})}\\
\nonumber
&\leq  \exp{\left(K (t-t_n)\right)} \limsup_{p\to\infty}\| \tn{Q}(t_n, \cdot) \|_{L^p(\RR^3; \RR^{3 \times 3})}\\
\nonumber
&\leq \exp{\left(K (t-t_n)\right)} \limsup_{p\to\infty}\| \tn{Q}(t_n, \cdot) \|_{L^2(\RR^3; \RR^{3 \times 3})}^{2/p}\| \tn{Q}(t_n, \cdot) \|_{L^\infty(\RR^3; \RR^{3 \times 3})}^{1-2/p}\\
\nonumber
&\leq \exp{\left(K (t-t_n)\right)} \limsup_{p\to\infty}M^{2/p}q_n^{1-2/p},
\end{align}
which yields the claim.

\centerline{\bf Acknowledgement}

The authors would like to thank the anonymous referee for many valuable suggestions and comments that helped to improve considerably the final version.




\begin{thebibliography}{10}

\bibitem{BerLio1}
H.~Berestycki and P.-L. Lions.
\newblock Nonlinear scalar field equations. {I}. {E}xistence of a ground state.
\newblock {\em Arch. Rational Mech. Anal.}, {\bf 82}(4):313--345, 1983.

\bibitem{BerLio2}
H.~Berestycki and P.-L. Lions.
\newblock Nonlinear scalar field equations. {II}. {E}xistence of infinitely
  many solutions.
\newblock {\em Arch. Rational Mech. Anal.}, {\bf 82}(4):347--375, 1983.

\bibitem{BerEdw}
A.~N. Beris and B.~J. Edwards.
\newblock {\em Thermodynamics of flowing systems with internal microstructure},
  volume~36 of {\em Oxford Engineering Science Series}.
\newblock The Clarendon Press, Oxford University Press, New York, 1994.
\newblock Oxford Science Publications.

\bibitem{DaiQSch}
M.~Dai, J.~Qing, and M.~Schonbek.
\newblock Asymptotic behavior of solutions to liquid crystal systems in {$\RR^3$}.
\newblock {\em Comm. Partial Differential Equations}, {\bf 37}(12):2138--2164,
  2012.

\bibitem{EF21}
E.~Feireisl.
\newblock On the long time behaviour of solutions to nonlinear diffusion
  equations on {$R^N$}.
\newblock {\em NoDEA}, {\bf 4}:43--60, 1997.

\bibitem{KirWiZa}
E.~Kirr, M.~Wilkinson, and A.~Zarnescu.
Dynamic Statistical Scaling in the Landau--de Gennes Theory of Nematic Liquid Crystals.
{\em J. Stat. Phys.},  {\bf 155}:625--657, 2014.

\bibitem{PaiZar2}
M.~Paicu and A.~Zarnescu.
\newblock Global existence and regularity for the full coupled
  {N}avier-{S}tokes and {$Q$}-tensor system.
\newblock {\em SIAM J. Math. Anal.}, {\bf 43}(5):2009--2049, 2011.

\bibitem{PaiZar1}
M.~Paicu and A.~Zarnescu.
\newblock Energy dissipation and regularity for a coupled {N}avier-{S}tokes and
  {$Q$}-tensor system.
\newblock {\em Arch. Ration. Mech. Anal.}, {\bf 203}(1):45--67, 2012.

\bibitem{Sch3}
M.~E. Schonbek.
\newblock Large time behaviour of solutions to the {N}avier-{S}tokes equations.
\newblock {\em Comm. Partial Differential Equations}, {\bf 11}(7):733--763,
  1986.

\bibitem{Sch2}
M.~E. Schonbek.
\newblock Lower bounds of rates of decay for solutions to the {N}avier-{S}tokes
  equations.
\newblock {\em J. Amer. Math. Soc.}, {\bf 4}(3):423--449, 1991.

\bibitem{Sch1}
M.~E. Schonbek.
\newblock Asymptotic behavior of solutions to the three-dimensional
  {N}avier-{S}tokes equations.
\newblock {\em Indiana Univ. Math. J.}, {\bf 41}(3):809--823, 1992.

\bibitem{Wil}
M. Wilkinson.
 \newblock Strictly physical global weak solutions of a Navier-Stokes Q-tensor system with singular potential.
 \newblock {\em Arch. Ration. Mech. Anal.}, {\bf 218}(1):487--526 (2015).


\end{thebibliography}

\def\cprime{$'$} \def\ocirc#1{\ifmmode\setbox0=\hbox{$#1$}\dimen0=\ht0
  \advance\dimen0 by1pt\rlap{\hbox to\wd0{\hss\raise\dimen0
  \hbox{\hskip.2em$\scriptscriptstyle\circ$}\hss}}#1\else {\accent"17 #1}\fi}



\end{document}